\documentclass[11pt]{article}
\usepackage{amssymb,amsmath,amsthm,amsfonts,amscd,latexsym}
\usepackage{mathrsfs}
\usepackage{graphicx}
\usepackage{epstopdf}
\usepackage{longtable}
\usepackage{color}

\renewcommand{\paragraph}{\roman{paragraph}}
\input cyracc.def

\newcommand{\F}{\mathbb{F}}

\newtheorem{thm}{Theorem}[section]
\newtheorem{lem}[thm]{Lemma}
\newtheorem{prop}[thm]{Proposition}
\newtheorem{Remark}[thm]{Remark}
\newtheorem{Def}[thm]{Definition}

\newtheorem{Cor}[thm]{Corollary}

\begin{document}
\title{\bf Sufficient conditions for STS$(3^k)$ of 3-rank $\leq 3^k-r$ to be resolvable}
\author{
Yaqi Lu, Minjia Shi\thanks{Yaqi Lu and Minjia Shi, School of Mathematical Sciences, Anhui University, Hefei, Anhui, 230601,
China, {\tt smjwcl.good@163.com, LyqSunshine8@163.com.}} 
}

\date{}
\maketitle

{\bf Abstract:} {Based on the structure of non-full-$3$-rank $STS(3^k)$ and the orthogonal Latin squares, we mainly give sufficient conditions for $STS(3^k)$ of $3$-rank $\leq 3^k-r$ to be resolvable in the present paper. Under the conditions, the block set of $STS(3^k)$ can be partitioned into $\frac{3^k-1}{2}$ parallel classes, i.e., $\frac{3^k-1}{2}$ $1$-$(v,3,1)$ designs. Finally, we prove that $STS(3^k)$ of 3-rank $\leq 3^k-r$ is resolvable under the sufficient conditions.}

{\bf Keywords:} Steiner triple system; Latin squares; Parallel classes; Resolvable.

{\bf MSC (2010) :} 05B15.
\section{Introduction}
\hspace{0.6cm}A Steiner triple system ($STS$ in short) is a pair $(\mathcal{P},\mathcal{B})$ \cite{Codes}, where $\mathcal{P}$ is a set of $v$ elements, called points, and $\mathcal{B}$ is a collection of distinct subsets of $\mathcal{P}$ of size $3$, called blocks, such that every subset of points of size $2$ is contained in precisely $1$ block. In fact, a $STS$ is a  $2$-$(v,3,1)$ design.

Whether before or now, many scholars are interested in $STS(v)$'s structure and solvability, and they obtained some results when $v$ takes a smaller value, such as $STS(7)$, $STS(9)$, $STS(31)$ and so on. In \cite{Lu1}  and \cite{Lu2}, Bryant and Horsley already gave Steiner triple systems without parallel classes, an infinite family of Steiner triple systems without parallel classes and the solvability of $STS(7)$, $STS(9)$, $STS(13)$, $STS(15)$, $STS(19)$ (see also in \cite{Lu3}) and $STS(21)$. In 1999, Lam et al., studied cyclically resolvable cyclic Steiner triple systems of order $21$ and $39$ \cite{Lu4}. Jungnickel et al. in \cite{Lu7} gave the classification of Steiner triple systems on $27$ points with $3$-rank $24$. On the other hand, the existence of resolvable Steiner Quadruple Systems was also studied in \cite{Lu8}. Recently, Jungnickel et al. \cite{Xu} gave the structure of non-full-$3$-rank STS. Moreover, Kirkman \cite{Lu5} proved that there exists a Steiner triple system of order $v$ if and only if $v\equiv 1$ or $3$ mod $6$ (see \cite{Lu6}) and $3^k \equiv 3$ mod $6$, that is, there exist $STS(3^k)$. Therefore, the existence of STS$(3^k)$ is guaranteed.

Motivated by the work listed above, a natural question is that: for any positive integer $k$, when $STS(3^k)$ is resolvable ? This paper is devoted to giving sufficient conditions for the solvability of $STS(3^k)$.

The material is organised as follows. The next section contains the preliminaries of $STS$, which are necessary for the rest of the paper. In Section 3, based on the theorem
of the structure of non-full-3-rank $STS$, we prove that $STS$ $(3^k)$ is resolvable under some conditions and find the sufficient conditions for its solvability.

\section{Preliminaries}

\begin{Def}
A $t$-$(v,k,\lambda)$ design, or briefly a $t$-design, is a pair $(\mathcal{P},\mathcal{B})$, where $\mathcal{P}$ is a set of $v$ elements, called points, and $\mathcal{B}$ is a collection of distinct subsets of $\mathcal{P}$ of size $k$, called blocks, such that every subset of points of size $t$ is contained in precisely $\lambda$ blocks.
\end{Def}

A $STS$ is a pair $(\mathcal{P},\mathcal{B})$ where $\mathcal{P}$ is a set of $v$ elements, called points, and $\mathcal{B}$ is a collection of distinct subsets of $\mathcal{P}$ of size $3$, called blocks, such that every subset of points of size $2$ is contained in precisely $1$ block.

\begin{lem}(Theorem 8.1.3 in \cite{Codes})
Let $(\mathcal{P},\mathcal{B})$ be a $t$-$(v,k,\lambda)$ design. Let $0\leq i\leq t$. Then $(\mathcal{P},\mathcal{B})$ is an $i$-$(v,k,\lambda_{i})$ design, where $\lambda_{i}=\lambda\frac{\left(\begin{matrix}
$v-i$\\
$t-i$
\end{matrix}\right)}{\left(\begin{matrix}
$k-i$\\
$t-i$
\end{matrix}\right)}=\lambda\frac{(v-i)(v-i-1)\cdots(v-t+1)}{(k-i)(k-i-1)\cdots(k-t+1)}$.
\end{lem}
By  Lemma 2.2, we know that $STS(v)$  is a $1$-$(v,3,\frac{v-1}{2})$ design.
\begin{Def}
$P$ is called a parallel class if $P\subset \mathcal{B}$ and $(S,P)$ is a $1$-$(v,3,1)$ design.
\end{Def}
\begin{Def}
$STS(v)$ is called resolvable if $\mathcal{B}$ can be partitioned into parallel classes.
\end{Def}
From Definitions 2.3 and 2.4, we know that $STS(v)$ is resolvable if and only if $STS(v)$ can be partitioned into $\frac{v-1}{2}$ $1$-$(v,3,1)$ designs.
\begin{Def}
A Latin square is an $n\times n$ array filled with $n$ different symbols, each occurring exactly once in each row and exactly once in each column.
\end{Def}
Let $A$, $B$, $C$ be three sets. From the definition of Latin square, we know that a Latin square which is denoted by $LS(A,B,C)$ is a set of triples $\{(x,y,z)|(x,y,z)\in A\times B\times C \}$ such that $\forall x\in A$, $y\in B$, there eixsts unique $z\in C$, $(x,y,z)\in LS(A,B,C)$. Similarly, $\forall x\in A$, $z\in C$, there eixsts unique $ y\in B$, $(x,y,z)\in LS(A,B,C)$, $\forall y\in B$, $z\in C$, there eixsts unique $x\in A$, $(x,y,z)\in LS(A,B,C)$.
\section{$STS(3^k)$ of $3$-rank $\leq3^k-r$}
In this section, $v=3^k$, that is, there exist $STS(v)$ of $3$-rank $\leq3^k-r$ for arbitrary $r$. By $V^v$, we denote the vector space of all
$v$-tuples over $\F_{3}$. Denote by $\mathcal{D}$ the set of subspaces of $V^v$, each including the all-one vector and being orthogonal to at least one $STS(v)$; denote
$\mathcal{D}_j=\{D\in \mathcal{D}:dim(D)=j + 1\}$.
\begin{lem} \label{lem2.1}
(Lemma $3.8$, \cite{Xu})
Let $S$ be an arbitary Steiner Triple system $STS(N)$ contained in the triple system $\mathcal{D}$. Then the block set $\mathcal{B}$ of $\mathcal{S}$ splits as follows:\\
(1) for all $i=1,\ldots, M$, a set $\mathcal{B}_i$ of $T(T-1)/6$ blocks such that $(\mathcal{G}_i,\mathcal{B}_i)$ is a Steiner triple systems $\mathcal{S}_i$ on the $T$ points in $\mathcal{G}_i$;\\
(2) for each line $l$ of the affine geometry $\sum=AG(k-t,3)$, a set $\mathcal{B}_l$ of $T^2=3^{2t}$ blocks
 forming a transversal design $TD[3;T]$ on the three groups determined by the points of $l$.

\end{lem}
\begin{prop}
STS$(9)$ of $3$-rank $\leq7$ is resolvable.
\end{prop}
\begin{proof}
By Lemma \ref{lem2.1}, given a subspace $D$ from $\mathcal{D}_1$, the set of $STS(9)$ orthogonal to D is in one-to-one correspondence with the collections of
$3$ Steiner triple systems of order $3$ and $1$ Latin square of order $3$.
A generator matrix for a subspace $D$ from $\mathcal{D}_1$ is
\begin{center} $
\left[\begin{array}{ccccccccc}
    1&1&1&1&1&1&1&1&1\\
    0&0&0&1&1&1&2&2&2\\
\end{array}\right].
$
\end{center}
Let the Latin square of order $3$ be
{\small\begin{longtable}{| c | c | c |}
\hline
1&2&3\\
\hline
2&3&1\\
\hline
3&1&2\\
\hline
\end{longtable}}
\hspace{-0.6cm}Then the indence matrix of $STS(9)$ is shown as follow:
\begin{center} $
I_{9}=\left[\begin{array}{ccccccccc}
    1&1&1&0&0&0&0&0&0\\
    0&0&0&1&1&1&0&0&0\\
    0&0&0&0&0&0&1&1&1\\
    1&0&0&1&0&0&1&0&0\\
    1&0&0&0&1&0&0&1&0\\
    1&0&0&0&0&1&0&0&1\\
    0&1&0&1&0&0&0&1&0\\
    0&1&0&0&1&0&0&0&1\\
    0&1&0&0&0&1&1&0&0\\
    0&0&1&1&0&0&0&0&1\\
    0&0&1&0&1&0&1&0&0\\
    0&0&1&0&0&1&0&1&0\\
\end{array}\right].
$
\end{center}
Therefore, the nonzero positions of each row of $I_9$ are $\mathcal{B}=\big\{\{1,2,3\},\{4,5,6\},\{7,8,9\},\{1,4,7\},\\\{1,5,8\},\{1,6,9\}, \{2,4,8\},\{2,5,9\},\{2,6,7\},\{3,4,9\},\{3,5,7\},\{3,6,8\}\big\}.$
Obviously, $STS(9)$'s block set $\mathcal{B}$ is partitioned into four parallel  classes, i.e., \big\{\{1,2,3\},\{4,5,6\},\{7,8,9\}\big\},\{\{1,4,7\},\{2,5,\\9\},\{3,6,8\}\big\}, \big\{\{1,5,8\},\{2,6,7\},\{3,4,9\}\big\},\big\{\{1,6,9\},\{2,4,8\},\{3,5,7\}\big\}. So $STS(9)$ of $3$-rank $\leq7$ is resolvable.
\end{proof}

\begin{prop}
STS$(27)$ of $3$-rank $\leq24$ is resolvable.
\end{prop}
\begin{proof}
The proof is similar to that of Proposition 3.2.
\end{proof}
\hspace{-0.5cm}In order to give sufficient conditions for solvability of $STS(3^k)$, we first introduce the concept of orthogonal Latin square.
\begin{Def}
A orthogonal Latin square of order $n$ over two sets $Q$ and $E$, each consisting of $n$ symbols, is an $n\times n$ arrangement of cells, each cell containing an ordered pair $(q,e)$, where $q$ is in $Q$ and $e$ is in $E$, such that every row and every column contains each element of $Q$ and each element of $E$ exactly once, and that no two cells contain the same ordered pair.
\end{Def}
\begin{Remark}
Every of two orthogonal Latin squares of order $n$ includes $n$ parallel classes.
\end{Remark}

\begin{prop}
STS$(9)$ of $3$-rank $\leq8$ is resolvable. Moreover, if there always exsits an orthogonal Latin square for all the Latin squares of order $9$ which appear in the decomposition of $STS(27)$, then STS$(27)$ of $3$-rank $\leq25$ is resolvable.
\end{prop}
\begin{proof}
By Lemma \ref{lem2.1}, given a subspace $D$ from $\mathcal{D}_1$, the set of $STS(27)$ orthogonal to D is in one-to-one correspondence with the collections of
$3$ $STS$ of order $9$ and $1$ Latin square of order $9$.
Let the (orthogonal) Latin square of order $9$ be \begin{center}
{\small\begin{longtable}{| c | c | c | c | c | c | c | c | c |}
\hline
$1\alpha$&$2\beta$&$3\gamma$&$4\delta$&$5\epsilon$&$6\varepsilon$&$7\zeta$&$8\theta$&$9\eta$\\
\hline
$2\gamma$&$3\alpha$&$1\beta$&$5\varepsilon$&$6\delta$&$4\epsilon$&$8\eta$&$9\zeta$&$7\theta$\\
\hline
$3\beta$&$1\gamma$&$2\alpha$&$6\epsilon$&$4\varepsilon$&$5\delta$&$9\theta$&$7\eta$&$8\zeta$\\
\hline
$4\zeta$&$5\eta$&$6\theta$&$7\beta$&$8\alpha$&$9\gamma$&$1\delta$&$2\epsilon$&$3\varepsilon$\\
\hline
$5\theta$&$6\zeta$&$4\eta$&$8\gamma$&$9\beta$&$7\alpha$&$2\varepsilon$&$3\delta$&$1\epsilon$\\
\hline
$6\eta$&$4\theta$&$5\zeta$&$9\alpha$&$7\gamma$&$8\beta$&$3\epsilon$&$1\varepsilon$&$2\delta$\\
\hline
$7\delta$&$8\epsilon$&$9\varepsilon$&$1\zeta$&$2\theta$&$3\eta$&$4\alpha$&$5\beta$&$6\gamma$\\
\hline
$8\varepsilon$&$9\delta$&$7\epsilon$&$2\eta$&$3\zeta$&$1\theta$&$5\gamma$&$6\alpha$&$4\beta$\\
\hline
$9\epsilon$&$7\varepsilon$&$8\delta$&$3\theta$&$1\eta$&$2\zeta$&$6\beta$&$4\gamma$&$5\alpha$\\
\hline
\end{longtable}}\vspace{-0.6cm}
\end{center}
The first part of the decomposition produces $4$ parallel classes, and the other part of the decomposition produces $9$ parallel classes since there exists an orthogonal Latin square for the Latin square of order $9$.
It means that $STS(27)$'s block set $\mathcal{B}$ is partitioned into $13$ parallel  classes. Therefore, $STS(27)$ of $3$-rank $\leq25$ is resolvable.
\end{proof}

All in all, if the $STS(3^k)$ is resolvable, then $\frac{3^k}{3\times3^{k-r+1}}=3^{r-2}$ Latin squares can be considered as a group to produce $3^{k-r+1}$ parallel classes in $STS(3^k)$ of $3$-rank $\leq3^k-r$. \\

The following theorem is the main result in this paper.
\begin{thm}
Let $S$ be an $STS(3^k)$ of $3$-rank at most $3^k-r$. If there always exsits an orthogonal Latin square for each Latin square of order $3^{k-r+1}$ that appears in the decomposition of $S$ and all STS$(3^{k-r+1})$ which appear in the decomposition of $S$ are resolvable, then $S$ is resolvable.
\end{thm}
\begin{proof}
By Lemma \ref{lem2.1}, given a subspace $D$ from $\mathcal{D}_{r-1}$, the set of $S$ orthogonal to D is in one-to-one correspondence with the collections of
$3^{r-1}$ $S$ of order $3^{k-r+1}$ and $\frac{3^{r-1}(3^{r-1}-1)}{6}$ Latin squares of order $3^{k-r+1}$.
Let $G_1=\{1,2,3,\cdots,3^{k-r+1}-1,3^{k-r+1}\}$, $G_2=\{3^{k-r+1}+1,3^{k-r+1}+2,3^{k-r+1}+3,\cdots,2\times3^{k-r+1}-1, 2\times3^{k-r+1}\}$, $\ldots$, $G_{3^{r-1}}=\{3^k-3^{k-r+1}+1,3^k-3^{k-r+1}+2,3^k-3^{k-r+1}+3,\cdots,3^k\}$, the first part $3^{r-1}$ STS$(3^{k-r+1})$ can be denoted by $STS(G_1)\cup STS(G_2)\cup\cdots\cup STS(G_{3^{r-1}})$, this $1$-$(3^k,3,\frac{3^{k-r+1}-1}{2})$ design is resolvable. For each $1\leq i\leq 3^{r-1}$, STS$(G_i)={G_i}^1\cup {G_i}^2 \cup {G_i}^3 \cup {G_i}^4\cup\cdots\cup {G_i}^{\frac{3^{k-r+1}-1}{2}}$, then $\cup_{i}{G_i}^1$, $\cup_{i}{G_i}^2$, $\cdots$, $\cup_{i}{G_i}^{\frac{3^{k-r+1}-1}{2}-1}$, and $\cup_{i}{G_i}^{\frac{3^{k-r+1}-1}{2}}$ are parallel classes. So the first part of the decomposition produces $\frac{3^{k-r+1}-1}{2}$ parallel classes.

 Let $C_p$ be $p$-th column of the check matrix of $S$, there exist $i,j,k$ such that $C_i+C_j+C_k=0$. Let $J_l=\{(i,j,k)\mid C_i+C_j+C_k=0\}$, it is easy to verify that $J_l$ is $S$. Thus $g=\frac{3^{r-1}(3^{r-1}-1)}{6}$ Latin squares of order $3^{k-r+1}$ can be denoted by \begin{equation}
\bigcup_{i,j,k}LS(H_i,H_j,H_k)=\bigcup_{l=1}^{\frac{g}{3^{r-2}}}\bigcup_{(i,j,k)\in J_l} LS(H_i,H_j,H_k),
\end{equation}$\cup_{(i,j,k)\in J_l} LS(H_i,H_j,H_k)$ is $1$-$(v,3,3^{k-r+1})$ design and it produces $3^{k-r+1}$ parallel classes, the other part of the decomposition produces $\frac{3^{r-1}(3^{r-1}-1)}{6}/3^{r-2}\times3^{k-r+1}=\frac{3^k-3^{k-r+1}}{2}$  parallel classes.

It means that the block of $S$ is partitioned into $\frac{3^k-1}{2}$ parallel classes. Therefore, we obtain that $S$ is resolvable.
\end{proof}

\begin{Cor}
The number $(k,t)$ of isomorphism classes of resolvable Steiner triple systems on $3^k$ points with $3$-rank exactly $3^k-k-1+t$, where $2\leq t\leq k-1$, satisfies $$(k,t)\geq \frac{{\widetilde{{N_1}}(T)}^M\cdot{\widetilde{{N_3}}(T)}^{M(M-1)/6}}{(T!)^M\cdot|AGL(k-t,3)|}
-\frac{{{{N_1}}(T')}^{M'}\cdot{{N_3}(T')}^{{M'}({M'}-1)/6}}{{((T')!)^{{M'}-k+t-2}}},$$
where $T=3^t$, $T'=3^{t-1}$, $M=3^{k-t}$ and $M'=3^{k-t+1}$; $N_1(T')$ is the number of $STS(T')$; $\widetilde{N_1}(T)$ is the number of resolvable $STS(T)$; $N_3(T')$ is the number of Latin squares of order $T'$; $\widetilde{N_3}(T')$ is the number of Latin squares of order $T'$ having an orthogonal mate.
\end{Cor}
\begin{proof}
The proof is essentially the same as for Theorem $4.4$ in \cite{Xu}, with the only difference that we count only resolvable $STS$.
\end{proof}

\section{Acknowledgement}
\hspace*{0.6cm} This research is supported by National Natural Science Foundation of China (61672036) and Excellent Youth Foundation of Natural Science Foundation of Anhui Province (1808085J20).


\begin{thebibliography}{99}\addtolength{\itemsep}{-7pt}
\bibitem{Lu1} Bryant D, Horsley D. Steiner triple systems without parallel classes. Siam Journal on Discrete Mathematics, 2015, 29(1):693-696.
\bibitem{Lu2} Bryant D, Horsley D. A second infinite family of Steiner triple systems without almost parallel classes. Journal of Combinatorial Theory, 2013, 120(7):1851-1854.
\bibitem{Lu6} Colbourn C J, Rosa A. Triple Systems. Clarendon Press, Oxford 1999.
\bibitem{Lu8}  Hartman A. The existence of resolvable steiner quadruple systems. Journal of Combinatorial Theory, 1987, 44(2):182-206.
\bibitem{Codes} Huffman W C,  Pless V. Fundamental of error correcting codes. United States of America by Cambridge University Press, New York www.cambridge.org 2003.
\bibitem{Lu7} Jungnickel D, Magliveras S S, Tonchev V D, Wassermann A. The classification of Steiner triple systems on $27$ points with $3$-rank 24. Designs, Codes and Cryptography, 2018:1-9, DOI:10.1007/s10623-018-0502-5.
\bibitem{Xu} Jungnickel D, Tonchev V D. Counting Steiner triple systems with classical parameters and prescribed rank. Journal of Combinatorial Theory, Series A 162(2019) 10-33.
\bibitem{Lu5}  Kirkman T P. On a problem in combinations. Cambridge and Dublin Mathematical Journal, 1847, 191-204.
\bibitem{Lu3} Kaski P, Ostergard P R J. The Steiner triple systems of order $19$. Mathematics of Computation, 2004, 73(248):2075-2092.
\bibitem{Lu4} Lam C W H,  Miao Y. Cyclically resolvable cyclic Steiner triple systems of order $21$ and $39$. Discrete Mathematics, 2000, 219(1-3):173-185.

\end{thebibliography}
\end{document}